\documentclass[11pt]{article}
\usepackage{fullpage}
\usepackage{amsthm}
\usepackage{amssymb}
\usepackage{amsmath}
\usepackage{graphicx}
\usepackage{mathrsfs}
\usepackage{amscd}
\usepackage{authblk}

\newcommand{\Z}{\mathbb{Z}}

\newcommand{\R}{\mathbb{R}}

\newcommand{\G}{\mathscr{G}}

\newcommand{\F}{\mathscr{F}}

\usepackage{relsize}
\newtheorem{theorem}{Theorem}

\newtheorem{proposition}[theorem]{Proposition}

\theoremstyle{definition}
\newtheorem*{conjecture}{Conjecture}

\title{On freeness of the random fundamental group}
\author{Andrew Newman \thanks{The Ohio State University}}
\date{\today}
\begin{document}
\maketitle
\begin{abstract}
Let $Y(n, p)$ denote the probability space of random 2-dimensional simplicial complexes in the Linial--Meshulam model, and let $Y \sim Y(n, p)$ denote a random complex chosen according to this distribution. In a paper of Cohen, Costa, Farber, and Kappeler, it is shown that for $p = o(1/n)$ with high probability $\pi_1(Y)$ is free. Following that, a paper of Costa and Farber shows that for values of $p$ which satisfy $3/n < p \ll n^{-46/47}$, with high probability $\pi_1(Y)$ is not free. Here we improve on both of these results to show that there are explicit constants $\gamma_2 < c_2 < 3$, so that for $p < \gamma_2/n$ with high probability $Y$ has free fundamental group and that for $p > c_2/n$, with high probability $Y$ has fundamental group which either is not free or is trivial.
\end{abstract}
\section{Introduction}
For positive integers $n$ and $d$ and $p = p(n) \in [0, 1]$, the space of Linial--Meshulam random $d$-dimensional simplicial complexes, first introduced in \cite{LM} and \cite{MW} and denoted $Y_d(n, p)$, is defined to be the probability space of $d$-dimensional simplicial complexes on $n$ vertices with complete $(d - 1)$-skeleton where each of the possible $\binom{n}{d + 1}$ $d$-dimensional faces is included independently with probability $p$. Here we are primarily interested in the $d = 2$ case and so we suppress the dimension parameter and write $Y(n, p)$ for $Y_2(n, p)$. Now the question of the fundamental group of $Y \sim Y(n, p)$  is nontrivial and  has been studied in \cite{BHK, CCFK, CF, HKP}. Additionally the series of papers \cite{AL, AL2, ALLM, LP}, study $Y_d(n, p)$ in the regime $p = c/n$. We will describe these results below, but we introduce two constants first introduced in \cite{ALLM} and \cite{AL} that are needed to state our main theorem. Let $\gamma_2 = (2x(1 - x))^{-1}$ where $x$ is the unique nonzero solution to $\exp(-\frac{1 - x}{2x}) = x$ and let $c_2 = \frac{-\log y}{(1 - y)^2}$ where $y$ is the unique root in $(0, 1)$ of $3(1 - y) + (1 + 2y) \log y = 0$. Here we build on the work of \cite{ALLM}, \cite{CF}, and \cite{LP} to prove the following result about the fundamental group of a random 2-complex. Note that most of the theorems stated here are asymptotic results and we use the phrase ``with high probability'', abbreviated w.h.p., to mean that a property holds for $Y \sim Y_d(n, p)$ with probability tending to 1 as $n$ tends to infinity.  The following theorem is the main result of this paper.
\begin{theorem}
 If $c < \gamma_2$ and $Y \sim Y(n, c/n)$, then with high probability $\pi_1(Y)$ is a free group and if $c > c_2$ and $Y \sim Y(n, c/n)$ then with high probability $\pi_1(Y)$ is not a free group. 
\end{theorem}
Now $\gamma_2$ is first defined in \cite{ALLM} and $c_2$ is first defined in \cite{AL} and approximations are computed as $\gamma_2 \approx 2.455407$ and $c_2 \approx 2.753806$.
\section{The lower bound}
In this section we prove the first part of theorem that for $c < \gamma_2$ one has $\pi_1(Y)$ is a free group with high probability for $Y \sim Y(n, c/n)$. This result will follow by adapting the argument of \cite{ALLM} used to prove the following result.
\begin{theorem}[$2$-dimensional case of Theorem 1.4 from \cite{ALLM}]
Let $\gamma_2$ be as above. If $c < \gamma_2$ then w.h.p.\ $Y \sim Y(n, c/n)$ is $2$-collapsible or contains $\partial \Delta_{3}$ as a subcomplex.
\end{theorem}
We first define what it means for a simplicial complex to be $d$-collapsible. For a $d$-dimensional simplicial complex $Y$, we say that a $(d-1)$-dimensional face $\tau$ is \emph{free} if it is contained in exactly one $d$-dimensional face $\sigma \in Y$. For a free $(d - 1)$-face $\tau$ an \emph{elementary collapse} at $\tau$ is defined to be the simplicial complex $Y'$ obtained from $Y$ be removing $\tau$ and the unique $d$-face $\sigma$ in $Y$ containing $\tau$. If there is a sequence of elementary collapses that removes all $d$-dimensional faces of $Y$ we say that $Y$ is $d$-collapsible. Observe that elementary collapses are homotopy equivalences, so if a 2-complex is 2-collapsible (to a graph) then in particular it has free fundamental group. Therefore theorem 2 above almost proves the lower bound except for the problem of tetrahedron boundaries. Note that it is impossible to rule out $\partial \Delta_3$ appearing as a subcomplex of $Y \sim Y(n, c/n)$ for any $c > 0$ since the expected number of copies of $\partial \Delta_{3}$ in $Y \sim Y(n, c/n)$ approaches a Poisson distribution with mean $c^{4}/24$. Additionally, \cite{ALLM} does not state any result about partial collapsibility in the presence of a few copies of $\partial \Delta_{3}$ and indeed it is not clear that any partial collapsibility result would hold. However such a result is not needed to imply that the fundamental group of $Y \sim Y(n, c/n)$ is free for $c < \gamma_2$ as we will see below.\\

Following the convention of \cite{ALLM} define a \textit{core} to be a $2$-dimensional simplicial complex in which every edge is contained in at least two faces. Also for a $2$-complex $Y$, let  $R(Y)$ denote the simplicial complex obtained by collapsing all the free edges of $Y$ and let $R_{\infty}(Y)$ denote the simplicial complex obtained after repeatedly collapsing at all free edges until no free edges remain. The two key results of \cite{ALLM} that we will use are the following.
\begin{theorem}[2-dimensional case of Theorem 4.1 from \cite{ALLM}]
For every $c > 0$ there exists a constant $\delta = \delta(c) > 0$ such that w.h.p.\ every core subcomplex $K$ of $Y \sim Y(n, c/n)$ with $f_2(K) \leq \delta n^2$ must contain the boundary of a tetrahedron.
\end{theorem}
\begin{theorem}[2-dimensional case of Theorem 5.3 from \cite{ALLM}]
Let $\delta > 0$ and $0 < c < \gamma_2$ be fixed and suppose $Y \sim Y(n, c/n)$. Then w.h.p.\ $f_2(R_{\infty}(Y)) \leq \delta n^2$.
\end{theorem}

Now to bound the probability that $\pi_1(Y)$ is not a free group for $Y \sim Y(n, c/n)$, we will bound the probability that $Y \sim Y(n, c/n)$ for $c < \gamma_2$ has a core which contains no tetrahedron boundary or has a pair of tetrahedron boundaries that are not face disjoint. This will be an upper bound to the probability that $\pi_1(Y)$ is not free by the following proposition.
\begin{proposition}
Let $Y$ be a $2$-dimensional simplicial complex. If every core of $Y$ contains a tetrahedron boundary and all the tetrahedron boundaries are face-disjoint then $\pi_1(Y)$ is free
\end{proposition}
\begin{proof}
Let $\tilde{Y}$ be the 3-dimensional simplicial complex obtained from $Y$ be adding a 3-simplex inside all the tetrahedron boundaries of $Y$. Now $\tilde{Y}^{(2)} = Y$ so $\pi_1(Y) = \pi_1(\tilde{Y})$. Now let $Z$ be obtained from $\tilde{Y}$ by collapsing at a free 2-dimensional face at every 3-dimensional face, such a collapse will remove all the tetrahedra from $\tilde{Y}$ as the tetrahedron boundaries in $Y$ are face disjoint so every tetrahedron in $\tilde{Y}$ has that all of its faces are free. Equivalently, $Z$ is obtained from $Y$ by deleting one face from every tetrahedron boundary of $Y$. Now collapsing at free faces is a homotopy equivalence so $\pi_1(Z) = \pi_1(\tilde{Y})$. Furthermore $Z$ is 2-collapsible. Indeed $Z$ has no cores as a core $K$ of $Z$ would be a core in $Y$ as well since $Z$ is obtained from $Y$ by removing faces. But every core of $Y$ contains a tetrahedron boundary and $Z$ has no tetrahedron boundaries. Since $Z$ does not have a core it must be 2-collapsible, otherwise deleting all the isolated edges of $R_{\infty}(Z)$ would give us a subcomplex of $Z$ that has no faces of degree zero or one, so such a subcomplex would be a core. Thus $Z$ is 2-collapsible, in particular $Z$ is homotopy equivalent to a graph so $\pi_1(Z)$ is a free group.
\end{proof}
Now we are ready to prove the first part of Theorem 1, that is for $c < \gamma_2$ and $Y \sim Y(n, c/n)$, $\pi_1(Y)$ is a free group with high probability.
\begin{proof}[Proof of lower bound on Theorem 1]
Let $c < \gamma_2$ be fixed and suppose $Y \sim Y(n, c/n)$, by proposition 5, the probability that $\pi_1(Y)$ is not free is bounded above by the sum of the probability that $Y$ contains tetrahedron boundaries that share a face and the probability that $Y$ has a core with no tetrahedron boundary. First it is easy to bound the probability that $Y$ contains tetrahedron boundaries that share a face. Two tetrahedron boundaries in a simplicial complex sharing a face must meet in exactly one face. Two tetrahedron boundaries meeting at one face is a simplicial complex with 5 vertices and 7 faces, the expected number of such subcomplexes in $Y \sim Y(n, p)$ is $O(n^5 p^7)$ which in this case is $O(c^7/n^2) = o(1)$. So by Markov's inequality the probability that there are tetrahedron boundaries in $Y$ that are not face disjoint is $o(1)$.\\

We will now use the two theorems from \cite{ALLM} above to show that the probability that $Y$ has a core with no tetrahedron boundary is $o(1)$. Let $\delta = \delta(c)$ be the $\delta$ given by Theorem 3.  Let $\F$ denote the collection of $2$-complexes on $n$ vertices containing a core with no tetrahedron boundary and let $\G$ denote the collection of 2-complexes on $n$ vertices for which all cores have size at most $\delta n^2$. Note that if $Y \notin \G$, then $f_2(R_{\infty}(Y)) > \delta n^2$ since cores are unaffected by elementary collapses, so $Pr(Y \notin \G) = o(1)$ by Theorem 4. Now we bound $Pr(Y \in \F)$.
\begin{eqnarray*}
Pr(Y \in \F) &=& Pr(Y \in \F \cap \G) + Pr(Y \in \F \setminus \G) \\
&\leq& Pr(Y \in \F \cap \G) + Pr(Y \notin \G) \\
&\leq& Pr(Y \in \F \cap \G) + o(1) \\
&\leq& Pr(Y \in \{X : X \text{ has a core $K$ with at most $\delta n^2$ faces and no $\partial \Delta_3$}\}) + o(1)
\end{eqnarray*}
Now by the choice of $\delta$ and Theorem 3, we know that the probability that $Y$ has a core which has at most $\delta n^2$ faces but no tetrahedron boundary is $o(1)$. Thus we have that $Pr(Y \in \F) = o(1)$ which completes the proof.
\end{proof}
\section{The upper bound}
We now turn our attention to proving that when $c > c_2$ and $Y \sim Y(n, c/n)$, with high probability $\pi_1(Y)$ is not a free group. In fact relevant results by Costa and Farber (\cite{CF}) will prove that the cohomological dimension is 2. We refer the reader to \cite{Brown} for background on group cohomology theory. The main result of \cite{CF} is the following:
\begin{theorem}[Theorem 2 of \cite{CF}]
Assume that $p \ll n^{-46/47}$, then w.h.p.\ a random 2-complex $Y \sim Y(n, p)$ is asphericable. That is the complex $Z$ obtained from $Y$ by removing one face from each tetrahedron of $Y$ is aspherical (i.e. the universal cover of $Z$ is contractible).
\end{theorem}
From here Costa and Farber prove the following result.
\begin{theorem}[Theorem 3B of \cite{CF}]
For any constants $c > 3$ and $0 < \epsilon < 1/47$ and $p$ satisfying $c/n < p < n^{-1 + \epsilon}$, the cohomological dimension of $Y \sim Y(n, p)$ equals $2$ with high probability. 
\end{theorem}
To prove our upper bound from Theorem 1 we will use the following result of Linial and Peled \cite{LP} to reduce the constant 3 in Theorem 7 to $c_2$, the argument will follow exactly the argument of Costa and Farber in their proof of Theorem 7, but with the current state-of-the-art (and best-possible, also by \cite{LP}) threshold for emergence of homology in degree 2 for a random 2-complex.
\begin{theorem}[Special case of Theorem 1.3 from \cite{LP}]
Suppose $c > c_2$, then w.h.p.\ $Y \sim Y(n, c/n)$ has
$$\dim H_2(Y; \R) = \Theta(n^2).$$
\end{theorem}
In $\cite{LP}$, the constant implicit in $\Theta(n^2)$ is given explicitly, but we do not need it here. We are now ready to prove the second part of theorem 1.
\begin{proof}[Proof of upper bound on Theorem 1]
Fix $c > c_2$ and suppose that $Y$ is a simplicial complex drawn from $Y(n, c/n)$. Now let $Z$ be obtained from $Y$ by removing one face from every tetrahedron boundary. With high probability $\pi_1(Y) = \pi_1(Z)$, and by theorem 6, $Z$ is aspherical. Therefore showing that $\beta_2(Z) \neq 0$ would imply that the cohomological dimension of $\pi_1(Y)$ is at least two. Now by theorem 8 we know that with high probability $\beta_2(Y) = \Theta(n^2)$. Also we have by a first moment argument that the expected number of tetrahedron boundaries is bounded above by $c^4/24$. Therefore by Markov's inequality with high probability $Y$ has no more than, say, $n$ tetrahedron boundaries. Now given a 2-dimensional simplicial complex, removing a face can drop $\beta_2$ by at most one. Therefore when we remove one face from from each tetrahedron boundary of $Y$ to obtain $Z$ we drop $\beta_2(Y)$ by at most $n$, then w.h.p.\ $\beta_2(Z) = \Theta(n^2) > 0$. Thus the cohomological dimension of $\pi_1(Y)$ is at least two (actually equality holds by theorem 6), in particular $\pi_1(Y)$ is not a free group.
\end{proof}
\section{Concluding Remarks}
The statement of Theorem 1 perhaps implicitly suggests a sharp threshold for the property that a random 2-complex has fundamental group which is not free. However, it is worth mentioning that the property that the fundamental group of a simplicial complex is free is not a monotone property, so it is not obvious at all that a sharp threshold should exist.  However, by theorem 1 and theorem 7, for $c_2/n < p < 3 \log n/n$ and $Y \sim Y(n, p)$, with high probability $\pi_1(Y)$ is not free. Combining this with a result from \cite{HKP} that for $p > (2 \log n + \omega(n))/n$ (with $\omega(n) \rightarrow \infty$ as $n \rightarrow \infty$) and $Y \sim Y(n, p)$, $\pi_1(Y)$ has property (T) with high probability, we have that for $p > c_2/n$, the fundamental group of $Y \sim Y(n, p)$ is with high probability free only if it is trivial as the only free group with property (T) is the trivial group. On the other hand, \cite{CCFK} proves that for $p = o(1/n)$, $Y \sim Y(n, p)$ collapses to a graph with high probability. Thus we do have at least a coarse threshold of $1/n$ for the fundamental group of a random 2-complex to be either not free or trivial. \\

It remains to discover the fundamental group for $Y(n, p)$ for $\gamma_2 \leq c \leq c_2$ and $p = c/n$. Right now, there does not seem to be enough evidence to establish a conjecture. There are the following three possibilites for what happens to the fundamental group of $Y(n, p)$ in this intermediate regime:
\begin{enumerate}
\item $\gamma_2/n$ is the sharp threshold for the fundamental group of $Y(n, p)$ to go from a free group to a non-free group.
\item $c_2/n$ is the sharp threshold for the fundamental group of $Y(n, p)$ to go from a free group to a non-free group.
\item Neither $\gamma_2/n$ nor $c_2/n$ is the sharp threshold for the fundamental group of $Y(n, p)$ to go from a free group to a non-free group.
\end{enumerate}
Any of these three would be interesting in their own way.\\

If (1) holds, then by \cite{CF}, $Y(n, c/n)$ in the regime $c \in (\gamma_2, c_2)$ has cohomological dimension equal to 2 and $Y(n, c/n)$ is asphericable. Thus when we remove a face from every tetrahedron we have that $Y(n, c/n)$ is a $K(G, 1)$ for a group $G$ which has cohomological dimension 2. However, the reason for it to have cohomological dimension 2 must be different than the reason for $\pi_1(Y(n, c/n))$ to have cohomological dimension 2 in the regime $c > c_2$. Indeed, in the regime $c > c_2$, $H_2(Y(n, c/n), \R) \neq 0$ is enough to imply that the cohomological dimension of the fundamental group is at least 2. For $c < c_2$ the second homology group of $Y(n, c/n)$ with coefficients in $\R$ is trivial, after the removal of a face from each tetrahedron boundary.\\

 Moreover, there is an apparent lack of torsion in $H_1(Y(n, c/n))$ if $c < c_2$. This has not been proved, but extensive experiments conducted in \cite{KLNP} provide evidence to support this, and \cite{LP2} state the following conjecture regarding torsion in homology: 
\begin{conjecture}[2-dimensional case of the conjecture from \cite{LP2}]
For every $p = p(n)$ such that $|np - c_2|$ is bounded away from 0, $H_{1}(Y_d(n, p); \Z)$ is torsion-free with high probability. 
\end{conjecture} 
Torsion in homology is observed experimentally for $p$ close to $c_d/n$. For more about this torsion see \cite{KLNP}. All of this is to say that it is likely that one will not be able to prove that the cohomological dimension of $\pi_1(Y(n, c/n)) \geq 2$ for $\gamma_2 < c < c_2$ by proving $H_2(Y(n, c/n), \Z/q\Z) \neq 0$ for some prime $q$. \\

On the other hand if (2) holds, then $Y(n, c/n)$ in the regime $\gamma_2 < c < c_2$ is homotopy equivalent to a wedge of circles after the removal of one face from each tetrahedron boundary. This follows from the fact that $Y(n, c/n)$ is asphericable and that an aspherical space which is a CW-complex is unique up to homotopy equivalence. Now for $c < \gamma_2$, as we show above following the results of \cite{ALLM}, $Y(n, c/n)$ is homotopy to a wedge of circles after the removal of a face from each tetrahedron boundary. However, this homotopy equivalence is given by a sequence of elementary collapses which reduces the complex to a graph. It is proved in \cite{AL2} that such a series of elementary collapses is not possible for $c > \gamma_2$. Furthermore, \cite{LP3} points out that in the regime $\gamma_2 < c < c_2$, $Y(n, c/n)$ is far from being 2-collapsible in the sense that a constant fraction of the faces must be deleted to arrive at a complex which is 2-collapsible. Thus $Y(n, c/n)$ in the regime $\gamma_2 < c < c_2$ if (2) holds would be homotopy equivalent to a wedge of circles, but not via the same type of homotopy equivalence which exists for smaller values of $c$. \\

In summary, regardless of whether the truth is (1) or (2), new techniques will almost certainly be required to prove which is correct. Of course, (3) is a possibility as well. Indeed it is possible that no sharp threshold exists for the property that $\pi_1(Y(n, p))$ as we discuss above. It could also be that there is some $c^* \in (\gamma_2, c_2)$ so that $c^*/n$ is the sharp threshold for $\pi_1(Y(n, p))$ to be non-free, or that within this intermediate regime there is a positive probability $Y(n, p)$ is free and a positive probability that it is not.

\bibliography{ResearchBibliography}
\bibliographystyle{amsplain}
\end{document}